\documentclass[11pt]{amsart}
\usepackage{amsmath}
\usepackage{amssymb}
\usepackage{amsfonts}
\usepackage{amsthm}
\usepackage{enumerate}
\usepackage[mathscr]{eucal}
\usepackage{graphicx}
\usepackage[pdftex, bookmarksnumbered, bookmarksopen, colorlinks, citecolor=blue, linkcolor=blue]{hyperref}

\newtheorem{theorem}{Theorem}[section]
\newtheorem{lemma}[theorem]{Lemma}

\theoremstyle{definition}

\newtheorem{proposition}[theorem]{Proposition}

\theoremstyle{remark}
\newtheorem{remark}[theorem]{Remark}

\numberwithin{equation}{section}

\def\intslash{\rlap{\kern  .32em $\mspace {.5mu}\backslash$ }\int}
\def\qsl{{\rlap{\kern  .32em $\mspace {.5mu}\backslash$ }\int_{Q_x}}}

\def\off{{\text{\rm off}}}

\def\R{{\mathbb R}}
\def\Z{{\mathbb Z}}
\def\M{{\mathcal M}}
\def\C{{\mathcal C}}

\def\A{{\mathcal A}}
\def\Q{{\mathcal Q}}

\def\S{\mathbf S}
\def\Q{\mathcal Q}

\def\pari{\partial}
\def\Ga{\Gamma}
\def\ga{\gamma}

\def\eg{{\it e.g. }}

\def\supp{{\text{\rm supp}}}

\def\inn#1#2{\langle#1,#2\rangle}

\def\rta{\rightarrow}

\def\card{\text{\rm card}}
\def\lc{\lesssim}

\def\del{\delta}             
\def\eps{\varepsilon}

\def\zet{\zeta}
\def\tet{\theta}

\def\lam{\lambda}            \def\Lam{\Lambda}

\def\si{\sigma}              
\def\vphi{\varphi}
             
              \def\Om{\Omega}
\def\fr{\frac}

\newcommand{\Be}{\begin{equation}}
\newcommand{\Ee}{\end{equation}}
\newcommand{\Bes}{\begin{equation*}}
\newcommand{\Ees}{\end{equation*}}
\newcommand{\Bsp}{\begin{split}}
\newcommand{\Esp}{\end{split}}
\newcommand{\Bm}{\begin{multline}}
\newcommand{\Em}{\end{multline}}
\newcommand{\Bea}{\begin{eqnarray}}
\newcommand{\Eea}{\end{eqnarray}}
\newcommand{\Beas}{\begin{eqnarray*}}
\newcommand{\Eeas}{\end{eqnarray*}}
\newcommand{\Benu}{\begin{enumerate}}
\newcommand{\Eenu}{\end{enumerate}}
\newcommand{\Bi}{\begin{itemize}}
\newcommand{\Ei}{\end{itemize}}
\begin{document}

\title[Noncommutative singular integral operators]{Weak type $(1,1)$ boundedness of noncommutative singular integral operators with rough kernels}

\author[]{Xudong Lai}
\address{Xudong Lai:
Institute for Advanced Study in Mathematics\\
Harbin Institute of Technology\\
Harbin
150001\\
China;
Zhengzhou Research Institute\\
Harbin Institute of Technology\\
Zhengzhou
450000\\
China}

\email{xudonglai@hit.edu.cn\ xudonglai@mail.bnu.edu.cn}





\begin{abstract}
In this paper, we establish the weak type $(1,1)$ boundedness of noncommutative singular integral operators with rough kernels.
To deal with the rough kernels, we develop some orthogonality of the kernels themselves in the pseudo-localization arguments for the good functions  and use the microlocal decomposition and $TT^*$ arguments for the bad functions.
\end{abstract}

\maketitle

\section{Introduction}
In the last twenty years, noncommutative Calder\'on--Zygmund theory has developed rapidly, demonstrating its potential applications in operator algebras, geometric group theory and other fields (see \eg \cite{Mei07}, \cite{Par09}, \cite{MP09}, \cite{Cad18}, \cite{CXY13}, \cite{CCP22}, \cite{HLX23}, \cite{Lai24}, \cite{Del26}).
The foundational framework of noncommutative Calder\'on--Zygmund theory has now been established. For example, the  $L_p$ ($1<p\leq\infty$) and weak type $(1,1)$ boundedness of the noncommutative Hardy-Littlewood maximal operator was  investigated by Mei \cite{Mei07} in 2007.
Subsequently, Parcet \cite{Par09} in 2009 established the weak type $(1,1)$ boundedness of noncommutative Calder\'on--Zygmund operators with kernels satisfying  Lipschitz regularity conditions.
Since then, whether the regularity condition on the kernel can be relaxed to the H\"ormander condition has remained an open problem. Recently the Lipschitz regularity condition on the kernel has been weakened to the $L_2$-Dini regularity condition by Cadilhac, Conde-Alonso and Parcet in \cite{CCP22} (independently obtained by Hong, the author and Xu in \cite{HLX23}).

On the other hand, there are many classical singular integral operators without regularity conditions  (see \eg \cite{Chr88}, \cite{CR88}, \cite{Hof88}, \cite{See96}, \cite{Tao1999}, \cite{STW04}, \cite{Tao04}). It is of interest to develop noncommutative Calder\'on--Zygmund theory with rough kernels. Following this line of research, in 2024 the author \cite{Lai24}  studied  the most fundamental operators: the noncommutative maximal averaging operators with rough kernels, which are defined by
\Bes
M_rf(x)=\fr{1}{|B(x,r)|}\int_{B(x,r)}\Om(x-y)f(y)dy,
\Ees
where $\Om$ is a homogeneous function defined on $\R^d\setminus\{0\}$ of degree zero, i.e.
$
\Om(r\tet)=\Om(\tet)\label{e:19Hom}
$
for any $r>0$ and $\tet\in\S^{d-1}$. It was shown in \cite{Lai24} that $\{M_r\}_{r>0}$ satisfies a noncommutative maximal weak type $(1,1)$ estimate if $\Om\in L(\log^+ L)^2(\S^{d-1})$.

This paper is the second part of this project on noncommutative Calder\'on--Zygmund theory with rough kernels.
The previous one \cite{Lai24} mainly uses the microlocal decomposition to deal with both the good and bad functions which arise from a kind of noncommutative Calder\'on--Zygmund decomposition.
These methods are not sufficient to establish the analogous boundedness for noncommutative singular integral operators with rough kernels.
In this paper, we develop some new tools to study their weak type $(1,1)$ boundedness.

To illustrate our main result, we introduce some basic notation. Let $\M$ be a semifinite von Neumann algebra equipped with a \emph{n.s.f.} trace $\tau$ and let $\A$ be the von Neumann algebra $L_\infty(\R^d)\bar\otimes\M$ equipped with
the \emph{n.s.f.} trace $\vphi$ defined by $\vphi(f)=\int_{\R^d}\tau(f(x))dx$. For $1\leq p<\infty$, define
$L_p(\mathcal{A})$ as the noncommutative
$L_p$ space associated to the pair
$(\A,\varphi)$. For more details about noncommutative $L_p$ spaces we refer to Section \ref{s:282}.
For a function $f\in L_p(\A)$, we consider a singular integral operator formally defined by
\Be\label{e:28Tom}
T  f(x)=\int_{\R^d}K(x-y)f(y)dy
\Ee
where $K(x)={\Om(x)}/{|x|^d}$,   $\Om$ is a homogeneous function defined on $\R^d\setminus\{0\}$ of degree zero and satisfies the cancellation condition $\int_{\S^{d-1}}\Om(\tet)d\sigma(\tet)=0$. Now we state our main result as follows.

\begin{theorem}\label{t:28}
Let $T$ be defined in \eqref{e:28Tom}. Suppose that $\Om\in L(\log^+ L)^{\frac52}(\S^{d-1})$. Then the operator $T$ is of noncommutative weak type $(1,1)$, i.e.\ for any $f\in L_1(\A)$ and $\lambda>0$,
\Bes
\lam\vphi(\chi_{(\lam,+\infty)}(|Tf|))\lc \C_\Om\|f\|_{L_1(\A)},
\Ees
where $\C_\Om$ is a constant depending only on the dimension $d$ and $\Om$ (see \eqref{e:19constantom} for its definition).
\end{theorem}

\begin{remark}
It was shown by Seeger \cite{See96} that the condition $\Om\in L\log^+L(\S^{d-1})$ guarantees the weak type $(1,1)$ boundedness of the classical (commutative) operator $T$.  Compared with this classical result, the condition $\Om\in L(\log^+L)^{\fr52}(\S^{d-1})$ appearing in our main theorem is stronger; nevertheless, it is natural in view of recent related results for the noncommutative maximal operator.
\end{remark}

There are two distinct noncommutative Calder\'on--Zygmund decompositions appearing in \cite{Par09} and \cite{CCP22}. The newer one in \cite{CCP22} is more efficient for handling Calder\'on--Zygmund operators whose kernels satisfy some regularity conditions, but it is more difficult to apply to our problem. In particular, it may not work within the \(TT^*\) argument, which is a standard tool in the study of rough singular integral operators.
We therefore use the original noncommutative Calder\'on--Zygmund decomposition from \cite{Par09}, following the approach used for the maximal averaging operator \(\{M_r\}_{r>0}\) in \cite{Lai24}. The main strategy in \cite{Lai24} is to study a linearized singular integral operator
\[
\tilde{T}f(x,z) = \sum_{j \in \mathbb{Z}} T_j f(x) \,\varepsilon_j(z),
\]
where \(T_j\)'s are the dyadic decomposition operators of \(T\) and \(\{\varepsilon_j\}_j\) is a Rademacher sequence on a probability space \((\mathfrak{m}, P)\). Roughly speaking, we employ a microlocal decomposition to treat both the good and bad functions. In particular, the orthogonal identity
\[
\Big\| \sum_{j \in \mathbb{Z}} \varepsilon_j a_j \Big\|_{L_2(L_\infty(\mathfrak{m}) \otimes \mathcal{A})}^2 = \sum_j \|a_j\|_{L_2(\mathcal{A})}^2
\]
plays a key role in the argument for both the good and bad functions.

In this paper, we study the operator $T = \sum_j T_j$; therefore the crucial orthogonal identity mentioned above is absent, which leads to several new difficulties. Our main strategy is as follows. For the good functions, we do not employ the microlocal decomposition. Instead, we develop the orthogonality from rough kernels themselves within a pseudo-localization argument. This is achieved via the Littlewood--Paley decomposition, where the cancellation condition $\int_{\mathbb{S}^{d-1}} \Omega(\theta) \, d\sigma(\theta) = 0$ is crucial. We refer to Section \ref{s:196} for further details.
For the bad functions, we adapt the microlocal decomposition, and the key point is to consider the following $L_2$ estimate of the off-diagonal terms
\[
\big\| \sum_j T_j^{n,s,v} p_{n-j} f p_{n-j+s} \big\|_{L_2(\mathcal{A})}^2
\]
(see Lemma \ref{L:28fnjs}). In particular, it is very hard to deal with the non-symmetric functions $p_{n-j} f p_{n-j+s}$, which involve two different scales of dyadic cubes: $\mathcal{Q}_{n-j}$ and $\mathcal{Q}_{n-j+s}$. To overcome this difficulty, we split the sum of $(T_j^{n,s,v})^* T_i^{n,s,v}$ into two cases: $|i-j| \le s$ and $|i-j| > s$. This split is not necessary in classical arguments (see \eg \cite{See96}) but turns out to be a critical step in the noncommutative setting. After that, we develop a full $TT^*$ argument to handle both cases to incorporate our noncommutative framework. We refer to Section \ref{s:195} for more details.  Incidentally, the  method presented in this paper  also gives  a slightly new  proof of weak type $(1,1)$ boundedness for $\{M_r\}_{r\geq0}$.

\medskip

\subsection*{Outline of the paper} In Section \ref{s:282}, we give some preliminaries on noncommutative analysis, mainly noncommutative $L_p$ spaces and some useful inequalities. We will introduce the noncommutative Calder\'on--Zygmund decomposition and finish the proof of our main theorem based on several propositions in Section \ref{s:283}. The proofs of propositions related to the good and bad functions are  presented in Section \ref{s:196} and Section \ref{s:195}, respectively.

\vskip0.24cm

\section {Preliminaries}\label{s:282}
\vskip0.24cm
\subsection{Basic Notation}Throughout this paper, we only consider the dimension $d\ge2$, and the letter $C$ stands for a positive constant that is independent of the essential variables, not necessarily the same in each occurrence. $A\lc B$ means $A\leq CB$ for some constant $C$. By the notation $C_\eps$  we mean that the constant depends on the parameter $\eps$. $A\approx B$ means that $A\lc B$ and $B\lc A$.
$\Z_+$ denotes the set of all nonnegative integers and $|x|$ denotes the $\ell_2$ norm. For $s\in\R_+$, $[s]$ denotes the integer part of $s$.
For any finite set $A$, we denote by $\card(A)$ or $\#(A)$ the number of elements in $A$.
Let $s\geq0$, we define $$\|\Om\|_{L(\log^+\!\!L)^s}:=\int_{\S^{d-1}}|\Om(\tet)|(1+\log^+(|\Om(\tet)|))^s d\si(\tet),$$
where $d\si(\tet)$ denotes the sphere measure of $\S^{d-1}$. When $s=0$, we use the standard notation $\|\Om\|_{1}:=\|\Om\|_{L(\log^+\!\!L)^0}$.

Define $\mathcal{F}f$ (or $\hat{f}$) and $\mathcal{F}^{-1}f$ (or $\check{f}$) as the Fourier transform and the inverse Fourier transform
of $f$ by
$$\mathcal{F}f(\xi)=\int_{\R^d} e^{-i\inn{x}{\xi}}f(x)dx,\ \ \ \ \mathcal{F}^{-1}f(\xi)=\fr{1}{(2\pi)^{d}}\int_{\R^d}e^{i\inn{x}{\xi}}{f(x)dx}.$$

\subsection{Noncommutative $L_{p}$-spaces}
Let $\M$ be a semifinite von Neumann algebra equipped with a normal semifinite faithful (\emph{n.s.f.} for short) trace $\tau$.
Denote by $\M_{+}$ the positive part of $\M$ and let $\mathcal{S}_{\M+}$  be the set of all $x\in\M_{+}$ whose support projection has finite trace. Let $\mathcal{S}_{\M}$ be the linear span of $\mathcal{S}_{\M+}$; then $\mathcal{S}_{\M}$ is a $w^{*}$-dense $\ast$-subalgebra of $\M$. Let $0< p<\infty$. For any $x\in\mathcal{S}_{\M}$, $|x|^{p}\in\mathcal{S}_{\M}$ and we set
$$\|x\|_{p}=\big(\tau(|x|^p)\big)^{1/p},\ \ x\in\mathcal{S}_{\M}.$$
Here $|x|=(x^{\ast}x)^{\frac{1}{2}}$ is the modulus of $x$. Define the noncommutative $L_{p}$-space associated with $(\M,\tau)$ as the completion of $(\mathcal{S}_{\M},\|\cdot \|_{p})$, and it is denoted by $L_{p}(\M)$. For convenience, we set $L_{\infty}(\M) = \M$ equipped with the operator norm $\|\cdot \|_{\M}$. Let $L_{p}(\M)_{+}$ denote the positive part of $L_{p}(\M)$.

Suppose that $\M\subset B(\mathcal{H})$ acts on a separable Hilbert space $\mathcal{H}$. Let $\M'$ be the commutant of $\M$. A closed densely defined operator on $\mathcal{H}$ is said to be affiliated with $\M$ when it commutes with every unitary operator $u$ in $\M'$. If $x$ is a densely defined self-adjoint operator on $\mathcal{H}$
and $x = \int_{\R} \lambda \hskip1pt d \gamma_x(\lambda)$ is its corresponding spectral
decomposition, then the spectral projection $\int_{\mathcal{I}} d
\gamma_x(\lambda)$ will simply be denoted by $\chi_{\mathcal{I}}(x)$, where $\mathcal{I}$ is a measurable subset of $\R$. A closed and densely defined operator
$x$ affiliated with $\mathcal{M}$ is called \emph{$\tau$-measurable} if
there exists $\lambda > 0$ such that $$\tau \big( \chi_{(\lambda,\infty)}
(|x|) \big) < \infty.$$
We denote the $\ast$-algebra of \emph{$\tau$-measurable} operators by $L_{0}(\M)$.
For $1\leq p<\infty$, the weak $L_{p}$-space $L_{p,\infty}(\M)$ is defined as the set of all $x$ in $L_0(\M)$ with finite quasi-norm
$$\|x\|_{p,\infty}=\sup_{\lambda > 0}\lambda\,\tau \big( \chi_{(\lambda,\infty)}
(|x|) \big)^{\frac{1}{p}}<\infty.$$
We refer the reader to \cite{PX03} for a detailed exposition of noncommutative $L_{p}$-spaces.

\subsection{Vector-valued noncommutative $L_{p}$-spaces}
We first recall the column space. Let $(\Sigma,\mu)$ be a measure space. The column space $L_p(\M;L^c_2(\Sigma))$ consists of the operator-valued functions $f$ with finite norm for $p\geq1$ (quasi-norm for $0<p<1$)
$$\|f\|_{L_p(\M;L^c_2(\Sigma))}=\Big\|\Big(\int_{\Sigma}f^*(\omega)f(\omega)d\mu(\omega)\Big)^{\frac12}\Big\|_{p}<\infty.$$
We refer the reader to \cite{PX03} for the precise definition and related properties of Hilbert-valued operator spaces.
The most important property for our purpose is the following Hilbert-valued H\"older type inequality (see \eg \cite[Proposition 1.1]{Mei07}).
\begin{lemma}\label{l:28holder}\rm
Let $0<p,q,r\leq\infty$ with $1/r=1/p+1/q$. Then for any $f\in L_p(\M;L^c_2(\Sigma))$ and $g\in L_q(\M;L^c_{2}(\Sigma))$
$$\Big\|\int_{\Sigma}f^*(\omega)g(\omega)d\mu(\omega)\Big\|_{r}\leq \Big\|\Big(\int_{\Sigma}|f(\omega)|^2d\mu(\omega)\Big)^{\frac12}\Big\|_{p}
\Big\|\Big(\int_{\Sigma}|g(\omega)|^2d\mu(\omega)\Big)^{\frac12}\Big\|_{q}.$$
\end{lemma}

We also need the following convexity inequality (or the Cauchy-Schwarz type inequality) for operator-valued functions (see \cite[Page 9]{Mei07}). Let $(\mathfrak{m},\mu)$ be a measure space. Suppose that $f:\ \mathfrak{m}\rta\M$ is a weak-$*$ integrable function and $g:\mathfrak{m}\rta\mathbb C$ is an integrable function. Then
\Be\label{e:28conv}
\Big|\int_{\mathfrak{m}}f(x)g(x)d\mu(x)\Big|^2\leq\int_{\mathfrak{m}}|f(x)|^2d\mu(x)\int_{\mathfrak{m}}|g(x)|^2d\mu(x).
\Ee

\subsection{Conditional expectations and martingale differences}
Let $\A$ be the von Neumann algebra $L_\infty(\R^d)\bar\otimes\M$ equipped with
the \emph{n.s.f.} trace $\vphi(f)=\int_{\R^d}\tau(f(x))dx$. For $1\leq p<\infty$, define
$L_p(\mathcal{A})$ as the noncommutative
$L_p$ space associated to the pair
$(\A,\varphi)$
with respect to the following norm
\Be\label{e:19deflp}
\|f\|_{L_p(\A)}= \Big( \int_{\R^d} \tau \,
\big( |f(x)|^p \big) \, dx \Big)^{\frac1p}.
\Ee
From \eqref{e:19deflp} we see that
$L_p(\mathcal{A})$ is isomorphic to the Bochner $L_p$ space with
values in $L_p(\mathcal{M})$. For convenience, we set $L_\infty(\A)=\A$ equipped with the operator
norm.
The lattices of projections are written as $\M_p$ and $\A_p$,
while $1_{\M}$ and $1_{\A}$ stand for the unit elements in $\M$ and $\A$. Let $L_p(\A)_+$ be the positive part of $L_p(\A)$. Notice that $L_2(\M)$ is a Hilbert space; then the following vector-valued Plancherel theorem will be frequently used in this paper
\Bes
\|\mathcal{F}f\|_{L^2(\A)}=(2\pi)^{\frac{d}{2}}\|f\|_{L^2(\A)}=(2\pi)^d\|\mathcal{F}^{-1}f\|_{L^2(\A)}.
\Ees

Let $\Q$ be the set of all dyadic cubes in $\R^d$. For any $Q\in\Q$, denote by $\ell(Q)$ the side length of the cube $Q$. Let $s Q$ be the cube with the same center as $Q$ such that $\ell(s Q)=s\ell(Q)$. Given an integer $k \in \Z$, $\Q_k$ will be defined as the set of dyadic cubes of side length $2^{-k}$. Let $|Q|$ be the volume of the cube $Q$. If $Q\in\Q$ and $f: \R^d \to
\M$ is integrable on $Q$, we define its average as
$f_Q = |Q|^{-1} \int_Q f(y) \, dy.$

For $k\in\Z$, set $\sigma_{k}$ as the $k$-th dyadic $\sigma$-algebra, i.e., $\sigma_{k}$ is generated by the dyadic cubes with side lengths equal to $2^{-k}$. Let $\mathsf{E}_k$ be the
conditional expectation associated with the classical dyadic
filtration $\sigma_{k}$ on $\R^d$. We also use $\mathsf{E}_k$ for the
tensor product $\mathsf{E}_k \otimes id_\M$ acting on $\A$. Then for  $1
\le p < \infty$ and $f \in L_p(\A)$, we have
$$\mathsf{E}_k(f) = \sum_{Q \in \Q_k}^{\null} f_Q \chi_Q,$$
where $\chi_Q$ is the characteristic function of $Q$.
Similarly, $\{\A_k\}_{k \in \Z}$ will stand for the corresponding
filtration, i.e.\ $\A_k = \mathsf{E}_k(\A)$. For simplicity, we write the conditional expectation $f_{k}:=\mathsf{E}_k(f)$ and the martingale difference $\Delta_{k} (f):=f_{k}-f_{k-1}=: df_{k}$.
\vskip0.24cm

\section {Proof of Theorem \ref{t:28}}\label{s:283}
\vskip0.24cm
The proof of Theorem \ref{t:28} is presented in this section, relying on several propositions. Their proofs are postponed to Sections \ref{s:196} and \ref{s:195}. We begin by introducing the noncommutative Calder\'on--Zygmund decomposition.

\subsection{Noncommutative Calder\'on--Zygmund decomposition}\label{s:1941}\quad
\vskip0.24cm

Since any $f$ can be expressed as a linear combination of four positive operators, it is enough to handle a positive $f$ belonging to $L_1(\A)$. A standard density argument allows us to restrict attention to the following dense subclass of $L_1(\A)_+$:
$$
\A_{c,+}=\{f:\R^d\rta\M\ | f\in\A_+,\ \text{$\overrightarrow{\rm {supp}}\,f$ is compact}\}.
$$
Here $\overrightarrow{\rm {supp}}\,f$ denotes the support of $f$ which is viewed as an operator-valued function on $\R^d$. That is, $\overrightarrow{\rm {supp}}\ f=\{x\in\R^d: \|f(x)\|_{\M}\neq0\}$. Suppose $\Om\in L(\log^+L)^{\frac52}(\S^{d-1})$. Define the constant
\Be\label{e:19constantom}
\mathcal{C}_\Om=\|\Om\|_{L(\log^+L)^{\frac52}}
+\int_{\S^{d-1}}|\Om(\tet)|\big(1+[\log^+({|\Om(\tet)|}/{\|\Om\|_{1}})]^{\frac52}\big)d\si(\tet),
\Ee
where $\log^+a=0$ for $0<a<1$ and $\log^+a=\log a$ for $a\geq1$. The finiteness of $\|\Om\|_{L(\log^+L)^{\frac52}}$ guarantees that $\mathcal{C}_\Om$ is finite. Now fix $f\in \A_{c,+}$ and define $f_k=\mathsf{E}_kf$ for each $k\in\Z$. Then $\{f_k\}_{k\in\Z}$ forms a positive dyadic martingale in $L_1(\A)$. Applying the well-known Cuculescu construction from \cite[Lemma 3.1]{Par09} at level ${\lam}{\C_\Om^{-1}}$ yields the following result.

\begin{lemma}\label{l:28cucu}
There exists a decreasing sequence $\{q_k\}_{k\in\Z}$, depending on $f$ and ${\lam}{\C_\Om^{-1}}$, with each $q_k$ a projection in $\A_p$ such that:
\begin{enumerate}[\rm (i).]
\item $q_k$ commutes with $q_{k-1}f_kq_{k-1}$ for all $k\in\Z$;
\item $q_k$ lies in $\A_k$ for every $k\in\Z$ and satisfies $q_kf_kq_k\leq{\lam}{\C_\Om^{-1}}q_k$;
\item Let $q=\bigwedge_{k\in\Z}q_k$. Then
\Bes
\vphi(1_\A-q)\leq\lam^{-1}\C_\Om\|f\|_{L_1(\A)};
\Ees
\item The projections $q_k$ can be expressed as follows: for some negative integer $m\in\Z$,
\Bes
q_k=\begin{cases}
1_\A\ \ &\qquad\text{if\ $k<m$,} \\
\chi_{(0,\lam\C_\Om^{-1}]}(f_k)\ &\qquad\text{if\ $k=m$},\\
\chi_{(0,\lam\C_\Om^{-1}]}(q_{k-1}f_kq_{k-1}) &\qquad\text{if\ $k>m$}.
\end{cases}
\Ees
\end{enumerate}
\end{lemma}

As in \cite{Par09}, we next give an alternative description of the projections $q_k$. This description will prove useful when estimating quantities involving $q_k$. Indeed, for every $k\in\Z$ we can write
$
q_k=\sum_{Q\in\mathcal{Q}_k}\xi_Q\chi_{Q},
$
where each $\xi_Q$ is a projection in $\M$ such that
\begin{enumerate}[\rm (i).]
\item An explicit formula holds (with $\widehat{Q}$ denoting the father dyadic cube of $Q$):
\Bes
\xi_Q=\begin{cases}
1_\M\ \ &\qquad\text{if $k<m$,} \\
\chi_{(0,\lam\C_\Om^{-1}]}(f_Q)\ &\qquad\text{if $k=m$},\\
\chi_{(0,\lam\C_\Om^{-1}]}(\xi_{\widehat{Q}}f_Q\xi_{\widehat{Q}}) &\qquad\text{if $k>m$}.
\end{cases}
\Ees
\item $\xi_Q\in\M_p$ and $\xi_Q\leq\xi_{\widehat{Q}}$;
\item $\xi_Q$ commutes with $\xi_{\widehat{Q}}f_Q\xi_{\widehat{Q}}$ and $\xi_Qf_Q\xi_Q\leq \lam\C_\Om^{-1}\xi_Q$.
\end{enumerate}

Set $p_k=q_{k-1}-q_k$. Using the more explicit representation above, we obtain
$
p_k=\sum_{Q\in\Q_k}(\xi_{\widehat{Q}}-\xi_Q)\chi_Q=:\sum_{Q\in\Q_k}{ p_Q\chi_Q}
$
with $ p_Q=\xi_{\widehat{Q}}-\xi_Q$.
It is readily verified that the $p_k$ are pairwise disjoint and satisfy $\sum_{k\in\Z}p_k=1_\A-q$.

The associated good and bad functions of $f$ are then defined by
\Bes
f=g+b,\quad g=\sum_{i,j\in\widehat{\Z}}p_if_{i\vee j}p_j,\quad b=\sum_{i,j\in\widehat{\Z}}p_i(f-f_{i\vee j})p_j,
\Ees
where we put $p_\infty=q$, $\widehat{\Z}=\Z\cup\{\infty\}$ and $i\vee j=\max (i,j)$. When $i$ or $j$ equals $\infty$, $i\vee j$ is understood to be $\infty$ and $f_\infty:=f$. From the linearity of $T$ we deduce
\Bes
{\vphi}(|Tf|>\lam)\leq{\vphi}(|Tg|>\lam/2)+{\vphi}(|Tb|>\lam/2).
\Ees

The proofs of estimates for the good and bad functions will be given separately. Before that, we state a lemma that provides a projection in $\A$ allowing us to reduce the problem to the case where all operators are restricted to that projection.
\begin{lemma}\label{l:19excep}
There exists a projection $\zet\in\A_p$ with the following properties:
\begin{enumerate}[\rm (i).]
\item $\lam\vphi(1_\A-\zet)\lc\C_\Om\|f\|_{L_1(\A)}$.
\item For any $Q_0\in\Q$ and any $x\in (2^{101}+1)Q_0$, we have $\zet(x)\leq1_\M-\xi_{\widehat{Q_0}}+\xi_{Q_0}$ and $\zet(x)\leq\xi_{Q_0}$.
\end{enumerate}
\end{lemma}

The proof of this lemma is a straightforward modification of \cite[Lemma 4.2]{Par09}. The specific constant $2^{101}+1$ is not crucial. We now turn our attention to the good functions.
\vskip0.24cm

\subsection{Estimates for the good functions}\label{s:1943}\quad
\vskip0.24cm

We first decompose $g$  into the diagonal terms and the off-diagonal terms
\Bes
g_d=qfq+\sum_{k\in\Z}p_kf_kp_k,\quad g_\off=\sum_{i\neq j}p_if_{i\vee j}p_j+qf(1_\A-q)+(1_\A-q)fq.
\Ees
The proofs for the diagonal term  $g_d$ and the off-diagonal term  $g_\off$ will be different, as we shall see below.

In the following we recall some properties of the diagonal term $g_d$ and the off-diagonal term $g_\off$ and related estimates which were proved  in \cite{Par09}.
\begin{lemma}\label{l:19gbasic}
For the diagonal term  $g_d$, we have the following basic property:
\Be\label{e:19gdes}
\|g_d\|_{L_1(\A)}\lc\|f\|_{L_1(\A)},\quad \|g_d\|_{L_\infty(\A)}\lc\lam\C_\Om^{-1}.
\Ee

Let $df_s$ be the martingale difference. For the off-diagonal term $g_{\off}$, we rewrite it as follows
\Bes
g_{\off}=\sum_{s\geq1}\sum_{k\in\Z}p_kdf_{k+s}q_{k+s-1}+q_{k+s-1}df_{k+s}p_k=:\sum_{s\geq1}\sum_{k\in\Z}g_{k,s}=:\sum_{s\geq1} g_{(s)}.
\Ees
The martingale difference sequence of $g_{(s)}$ satisfies $(dg_{(s)})_{k+s}=g_{k,s}$.
Meanwhile, we have the  estimate
\Be\label{e:19gbasic}
\sup_{s\geq1}\|g_{(s)}\|_{L_2(\A)}^2=\sup_{s\geq1}\sum_{k\in\Z}\|g_{k,s}\|_{L_2(\A)}^2\lc\lam\C_\Om^{-1}\|f\|_{L_1(\A)}.
\Ee
\end{lemma}

We first consider the diagonal terms, which are simpler since they behave similarly to those in the classical Calder\'on--Zygmund decomposition.
Following the classical strategy, we need to establish the $L_p$ boundedness of $T$ for some $p\in(1,\infty)$. In this situation, the condition on the kernel $\Om$ can actually be relaxed to $\Om\in L \log^+L (\S^{d-1})$.
\begin{lemma}\label{l:19tgl2}
Suppose that $\Om$ is homogeneous of degree zero, $\Om\in L \log^+ L (\S^{d-1})$ and satisfies the cancellation property $\int_{\S^{d-1}}\Om(\tet)d\si(\tet)=0$. Then for all $1<p<\infty$, we have
\Bes
\|Tf\|_{L_p(\A )}\lc\|\Om\|_{L\log^+L}\|f\|_{L_p(\A)}.
\Ees
\end{lemma}

Lemma \ref{l:19tgl2} holds for more general Bochner $L_p$ spaces with values in a UMD space. In fact, if $X$ is a UMD space, Burkholder \cite{Bur83} used the method of rotation to show that $T$ is bounded on $L^p(\R^d,X)$ ($1<p<\infty$) when $\Om$ is odd with $\Om\in L^1(\S^{d-1})$ or $\Om$ is even with $\Om\in L\log^+L(\S^{d-1})$, where $L^p(\R^d,X)$ is the Bochner $L_p$ space with values in $X$. As a consequence, since $L_p(\mathcal A)=L_p(\R^d,L_p(\M))$ and $L_p(\M)$ is a UMD space when $1<p<\infty$ (see \cite{PX03}), $T$ extends to a bounded operator on $L_p(\mathcal A)$.

Based on the above lemma, we can prove the estimate for the diagonal term $g_d$ as follows:
\Bes
\begin{split}
\vphi (|Tg_d|>\lam/4)&\lc\lam^{-2}\|Tg_d\|^2_{L_2(\A )}
\lc\lam^{-2}\C_\Om^2\|g_d\|^2_{L_2(\A)}\\
&\lc\lam^{-2}\C_\Om^2\|g_d\|_{L_1(\A)}\|g_d\|_{L_\infty(\A)}\lc\lam^{-1}\C_\Om\|f\|_{L_1(\A)},
\end{split}
\Ees
where in the first inequality we use  the Chebyshev  inequality, the second inequality follows from Lemma \ref{l:19tgl2} for $p=2$, and the third and fourth inequalities follow from \eqref{e:19gdes}.

In the remaining part of this subsection, we focus on the estimate for $g_\off$. We first use Lemma \ref{l:19excep} to reduce the proof to the case $\zet T g_\off \zet$, which will also be argued for the bad functions later in Section \ref{s:1942}.
Indeed,
\Bes
Tg_\off=(1_\A-\zeta)Tg_\off(1_\A-\zet)+\zet Tg_\off(1_\A-\zet)+(1_\A-\zet)Tg_\off\zet+\zet Tg_\off\zet.
\Ees
By property (i) in Lemma \ref{l:19excep}, we obtain
\Bes
\begin{split}
\vphi (|Tg_\off|>\lam/4)&\lc\vphi(1_\A-\zet)+\vphi (|\zet Tg_\off\zet|>\lam/8)\\
&\lc\lam^{-1}\C_\Om\|f\|_{L_1(\A)}+\vphi (|\zet Tg_\off\zet|>\lam/8).
\end{split}
\Ees
Thus our goal is to prove
\Be\label{e:19goffg}
\vphi(|\zet T g_\off \zet|>\lam/8)\lc\lam^{-1}\C_\Om\|f\|_{L_1(\A)}.
\Ee

In the following, we make a dyadic decomposition of the kernel $K$. Let $\phi$ be a smooth function on $\mathbb{R}^d$ which is supported in the annulus $\{2^{-1} < |x| < 2 \}$ and satisfies the partition of unity condition
\begin{equation*}
\sum_{j\in \mathbb{Z}}\phi_j(x) = 1 \quad \text{for all } x \in \mathbb{R}^d \setminus \{0\},
\end{equation*}
where $\phi_j(x) = \phi(2^{-j}x)$. We define the associated dyadic operator
\begin{align*}
T_j f(x) = \int_{\mathbb{R}^d}{K}_j(x-y)f(y) dy
\end{align*}
with the kernel ${K}_j(x) = \phi_j(x)\frac{\Omega(x)}{|x|^d}$. Then we have $T=\sum_{j\in\Z}T_j$.

By the above dyadic decomposition, the expression of $g_\off$ in Lemma \ref{l:19gbasic} and the formula $f=\sum_{n\in\Z}df_n$, we can write
\Bes
\begin{split}
\zet T g_\off\zet&=\zet\sum_{s\geq1}\sum_{j\in\Z} T_jg_{(s)}\zet=\zet\sum_{s\geq1}\sum_{j\in\Z}\sum_{n\in\Z} T_jd\big(g_{(s)}\big)_{n-j+s}\zet\\
&=\zet\sum_{s\geq1}\sum_{j\in\Z}\sum_{n\in\Z}   T_jg_{n-j,s}\zet
=\zet\sum_{s\geq1}\sum_{n\geq100}\sum_{j\in\Z}   T_jg_{n-j,s}\zet
\end{split}
\Ees
where in the last equality we apply the following observation: if $\zet(x)\neq0$ and $n<100$,
we have for all $s\geq1$,
\Be\label{e:28exec}
\begin{split}
\zet(x)T_jg_{n-j,s}(x)\zet(x)
&=\sum_{Q\in\Q_{n-j}}\zet(x)\chi_{((2^{101}+1)Q)^c}(x)\int_Q K_j(x-y)\\
&\quad\times\big[ p_{Q}(df_{n-j}) p_{n-j+s-1}
+ p_{n-j+s-1}(df_{n-j}) p_{Q}\big]dy\,\zet(x)\\
&=0
\end{split}
\Ee
here in the first equality we use $\zet(x) p_{Q}=0$ if $x\in (2^{101}+1)Q $ by property (ii) of $\zet$ in Lemma \ref{l:19excep},
and the second equality follows from the fact that $x\in ((2^{101}+1)Q )^c$ and $y\in Q$ implies $|x-y|\geq2^{100+j-n}$,
which contradicts the support of $K_j$ when $n<100$.

By  the Chebyshev  inequality, the triangle inequality and   $\zet$ is a projection in $\A$, we then obtain
\Bes
\begin{split}
\vphi (|\zet T g_\off \zet|>\lam)&\lc\lam^{-2}\|\zet T g_\off \zet\|_{L_2(\A )}^2\\
&\lc\lam^{-2}\bigg(\sum_{s\geq1}\sum_{n\geq100}\big\|\sum_j T_j g_{n-j,s}\big\|_{L_2(\A)}\bigg)^2.
\end{split}
\Ees
Hence, to finish the proof of the off-diagonal term $g_\off$, it is sufficient to show the following proposition.

\begin{proposition}\label{l:28boff}
For any $s\geq 1$ and $n\geq100$, there exists a positive constant $D(s,n)$ depending on $s$ and $n$ such that
\Be\label{e:19boffgoal}
\big\|\sum_j  T_jg_{n-j,s}\big\|^2_{L_2(\A)}
\lc D(s,n)\lam\|f\|_{L_1(\A)},
\Ee
and
\Be\label{e:19boffdecay}
\Big(\sum_{s\geq1}\sum_{n\geq100}D(s,n)^{\frac12}\Big)^2\lc\C_\Om.
\Ee
\end{proposition}
The proof of Proposition \ref{l:28boff} will be given in Section \ref{s:196}.
\vskip0.24cm

\subsection{Estimates for the bad functions}\label{s:1942}\quad
\vskip0.24cm

We split $Tb$ into four components:
\Bes
(1_\A-\zeta)Tb(1_\A-\zet)+\zet Tb(1_\A-\zet)+(1_\A-\zet)Tb\zet+\zet Tb\zet.
\Ees

By Lemma \ref{l:19excep} and the same reduction argument used for the good functions, it suffices to show that the desired estimate holds for $\vphi (|\zet Tb\zet|>\lam/4)$. Recall the expression of the bad function:
\Bes
b=\sum_{k\in\Z}p_k(f-f_k)p_k+\sum_{s\geq1}\sum_{k\in\Z}p_k(f-f_{k+s})p_{k+s}+p_{k+s}(f-f_{k+s})p_k
=:\sum_{s=0}\sum_{k\in\Z}b_{k,s},
\Ees
where
\Be\label{e:19db}
b_{k,0}=p_k(f-f_k)p_k,\quad b_{k,s}=p_k(f-f_{k+s})p_{k+s}+p_{k+s}(f-f_{k+s})p_k.
\Ee

We now list some basic properties of the bad functions that will be needed in the proof.
\begin{lemma}\label{l:19bbasic}
Let $b_{k,s}$ be  defined in \eqref{e:19db}. For any fixed $s\geq0$, the following hold:
\begin{enumerate}[\rm (i).]
\item An $L_1$ bound estimate:
$
\sum_{k\in\Z}\|b_{k,s}\|_{L_1(\A)}\lc\|f\|_{L_1(\A)};
$
\item The cancellation condition: for all $k\in\Z$ and every $Q\in\Q_{k+s}$,
$\int_{Q}b_{k,s}(y)dy=0$.
\end{enumerate}
\end{lemma}

The proof of Lemma \ref{l:19bbasic} can be found in \cite{Cad18} and \cite{Par09}.

From the definition of $T$, we further rewrite $Tb$ as follows: for any $x\in\R^d$,
\Bes
\begin{split}
\zet(x)Tb(x)\zet(x)&=\zet(x)\sum_{j\in\mathbb{Z}}T_j\biggl[\sum_{s\geq0}\sum_{n\in\mathbb{Z}}b_{n-j,s}\biggr](x)\zet(x)
  =\sum_{s\geq0}\sum_{n\in\mathbb{Z}}\sum_{j\in\mathbb{Z}}\zet(x)T_jb_{n-j,s}(x)\zet(x)\\
  &=\sum_{s\geq0}\sum_{n\geq100}\sum_{j\in\mathbb{Z}}\zet(x)T_jb_{n-j,s}(x)\zet(x),
\end{split}
\Ees
where the last equality follows because when $\zet(x)\neq0$ and $n<100$, we have
$$\zet(x)T_jb_{n-j,s}(x)\zet(x)=0,$$
by an argument analogous to that in \eqref{e:28exec}.

Consequently, to complete the treatment of the bad functions, it is enough to establish
\begin{equation}\label{e:19boff}
\vphi \big(\big|\zet\sum_{n\geq100}\sum_{s\geq0}\sum_{j\in
\mathbb{Z}}T_jb_{n-j,s}  \zet\big|>\lambda/4\big)\lc {\lambda}^{-1}\C_\Om\|f\|_{L_1(\A)}.
\end{equation}

Several important decompositions are crucial for the proof of \eqref{e:19boff}. They are stated as propositions below, and their proofs will be given in Section \ref{s:195}.

The first proposition indicates that \eqref{e:19boff} remains true when $\Om$ is restricted to a certain subset of $\S^{d-1}$. More precisely, for fixed $n\ge100$ and $s\geq 0$, we define $D^\iota=\{\theta\in\S^{d-1}:\,|\Om(\theta)|\geq2^{\iota (n+s)}\|\Om\|_{1}\}$, where $\iota\in (0,1)$ will be chosen later.
Let $T_{j,\iota}^{n,s}$ be given by
\Be\label{e:19tjiota}
T_{j,\iota}^{n,s}h(x)=\int_{\R^d}\Om\chi_{D^{\iota}}\Big(\frac{x-y}{|x-y|}\Big)
\cdot K_j(x-y)\cdot h(y)dy.
\Ee

\begin{proposition}\label{l:19cur}
Assume $\Om\in L(\log^+ L)^{\frac52}(\S^{d-1})$. With all the above notation, we have
\Bes
\vphi \big(\big|\zet\sum_{n\geq100}\sum_{s\geq0}\sum_{j\in
\mathbb{Z}}T_{j,\iota}^{n,s}b_{n-j,s}  \zet\big|>\lambda/8\big)\lc {\lambda}^{-1}\C_\Om\|f\|_{L_1(\A)}.
\Ees
\end{proposition}

Thus, by Proposition \ref{l:19cur}, it remains to verify (\ref{e:19boff}) under the additional assumption that for each fixed $n\geq100$ and $s\geq0$, the kernel $\Omega$ satisfies $\|\Om\|_{\infty}\leq2^{\iota (n+s)}\|\Om\|_{1}$ inside every $T_j$.

Next we introduce a \emph{microlocal decomposition} of the kernel. To this end, we construct a partition of unity on $\S^{d-1}$.
Let $k\geq100$. Choose a family $\{e^k_v\}_{v\in\Theta_k}$ of unit vectors on $\S^{d-1}$
satisfying:

(a)\ $|e^k_v-e^k_{v'}|\geq 2^{-k\ga-4}$ whenever $v\neq v'$;

(b)\ For every $\theta\in \S^{d-1}$, there exists an $e^k_v$ such that $|e^k_v-\theta|\leq 2^{-k\ga-3}$.

\noindent The constant $0<\ga<1$ in (a) and (b) will be fixed later. Such a collection can be obtained by taking a maximal set satisfying (a); then (b) follows automatically from maximality. Note that the cardinality of $\{e^k_v\}_v$ is  $C2^{k\ga(d-1)}$. For each $\tet\in\S^{d-1}$, only finitely many $e^k_v$ satisfy $|e^k_v-\tet|\leq2^{-k\ga-4}$. We next build an associated partition of unity on $\S^{d-1}$. Let $\eta$ be a smooth, nonnegative, radial function with $\eta(u)=1$ for $|u|\leq \frac12$ and $\eta(u)=0$ for $|u|>1$. Define
$$\tilde{\Ga}^k_v(u)=\eta\Big(2^{k\ga}\big(\frac{u}{|u|}-e^k_v\big)\Big),\quad
\Ga^k_v(u)=\tilde{\Ga}^k_v(u)\Big(\sum\limits_{v\in\Theta_k}\tilde{\Ga}^k_v(u)\Big)^{-1}. $$
It is easy to see that each $\Ga^k_v$ is homogeneous of degree $0$, and $\sum_{v\in\Theta_{k}}\Ga^k_v(u)=1$ for all $u\neq0$ and all $k$.
Now define the operator $T_j^{n,s,v}$ by
\Be\label{e:19Tjnsv}
T_j^{n,s,v}h(x)=\int_{\R^d}\Om(x-y)\cdot \Ga^{n+s}_v(x-y)\cdot K_j(x-y)\cdot h(y)dy.
\Ee
Clearly $T_j=\sum_{v\in\Theta_{n+s}}T_j^{n,s,v}.$

To separate directions in frequency space, we define a Fourier multiplier operator by
$$\widehat{G_{k,v}h}(\xi)=\Phi(2^{k\ga}\inn{e^k_v}{{\xi}/{|\xi|}})\hat{h}(\xi),$$
where $\hat{h}$ is the Fourier transform of $h$ and $\Phi$ is a smooth, nonnegative, radial function satisfying $0\leq\Phi(x)\leq1$, $\Phi(x)=1$ for $|x|\leq2$, and $\Phi(x)=0$ for $|x|>4$.
We then split $T_j^{n,s,v}$ as
$
T_j^{n,s,v}=G_{n+s,v}T_j^{n,s,v}+({\rm{I}}-G_{n+s,v})T_j^{n,s,v}.
$

The next proposition provides an $L^2$ estimate for the part $G_{n+s,v}T_j^{n,s,v}$.

\begin{proposition}{\label{l:19L^2}}
Let $n\geq100$ and $s\geq0$. Assume that within each $T_j$, we have $\|\Om\|_{\infty}\leq2^{\iota (n+s)}\|\Om\|_{1}$. Then  with all the above notation,
$$\Big\|\sum\limits_{v\in\Theta_{n+s}}\sum\limits_j G_{n+s,v}T_j^{n,s,v}b_{n-j,s}\Big\|^2_{L_2(\A)}
\lc(s+1)2^{-(n+s)\ga+2(n+s)\iota}\lam\C_\Om\|f\|_{L_1(\A)}.$$
\end{proposition}

For the remaining part $({\rm{I}}-G_{n+s,v})T_j^{n,s,v}$, the following estimate was proved in \cite[Page 1451, Lemma 3.5]{Lai24}.
\begin{proposition}\label{l:19L^1}
Set $L^{n,s,v}_j=({\rm{I}}-G_{n+s,v})T_j^{n,s,v}$. For $n\geq100$, $s\geq0$, and under the condition $\|\Om\|_{\infty}\leq2^{\iota (n+s)}\|\Om\|_{1}$ in each $T_j$, there exists a positive constant $\alpha$ such that
$$\sum\limits_j\sum_{v\in\Theta_{n+s}}\big\|L_j^{n,s,v}b_{n-j,s}\big\|_{L_1(\A)}\lc2^{-(n+s)\alpha}\C_\Om\|f\|_{L_1(\A)}.$$
\end{proposition}

We now conclude the proof of \eqref{e:19boff}. It is sufficient to establish \eqref{e:19boff} under the assumption that for all fixed $n\geq100$ and $s\geq0$, we have $\|\Om\|_{\infty}\leq2^{\iota (n+s)}\|\Om\|_{1}$ inside $T_j$.  By  the Chebyshev inequality and the triangle  inequality, we obtain
\begin{equation*}
\begin{split}
&\quad \vphi \big(\big|\zet\sum_{n\geq100}\sum_{s\geq0}\sum_{j\in
\mathbb{Z}}T_jb_{n-j,s}  \zet\big|>\lambda/8\big)\\
&\lc {\lam^{-2}}\Big\|\zet\sum\limits_{n\geq100}\sum_{s\geq0}\sum\limits_{j}\sum\limits_{v\in\Theta_{n+s}}G_{n+s,v}T_j^{n,s,v}b_{n-j,s}  \zet\Big\|_{L_2(\A)}^2
\\
&\quad +\lam^{-1}\sum_{n\geq100}\sum_{s\geq0}\sum\limits_j\sum_{v\in\Theta_{n+s}}\big\|\zet L_j^{n,s,v}b_{n-j,s}  \zet\big\|_{L_1(\A)}\\
&=:I+II.
\end{split}
\end{equation*}

Choose $0<\iota<\frac{\ga}{2}<\fr{1}{2}$. Applying the triangle  inequality, then the H\"older inequality to eliminate $\zet$ (since $\zet$ is a projection), and finally Proposition \ref{l:19L^2}, we get
\begin{equation*}
\begin{split}
I&\lc\lam^{-2}\Big(\sum\limits_{n\geq100}\sum_{s\geq 0}\Big\|\zet\sum\limits_{j}
\sum\limits_{v\in\Theta_{n+s}}G_{n+s,v}T_j^{n,s,v}b_{n-j,s}\zet\Big\|_{L_2(\A)}\Big)^2\\
&\lc\lam^{-2}\Big(\sum\limits_{n\geq100}\sum_{s\geq 0}\big((s+1)2^{-(n+s)\ga+2(n+s)\iota}\C_\Om\lam\|f\|_{L_1(\A)}\big)^{\frac{1}{2}}\Big)^{2}
\lc \lam^{-1}\C_\Om\|f\|_{L_1(\A)}.
\end{split}
\end{equation*}

For $II$, using  the H\"older inequality to remove $\zet$ and then Proposition \ref{l:19L^1} yields
\Bes
\begin{split}
II&\lc\lam^{-1}\sum_{n\geq100}\sum_{s\geq0}\sum\limits_j\sum_{v\in\Theta_{n+s}}\big\|L_j^{n,s,v}b_{n-j,s}\big\|_{L_1(\A)}\\
&\lc\lam^{-1}\sum_{n\geq100}\sum_{s\geq0}2^{-(n+s)\alpha}\C_\Om\|f\|_{L_1(\A)}\lc\C_\Om\lam^{-1}\|f\|_{L_1(\A)}.
\end{split}
\Ees

Thus we have established \eqref{e:19boff} by means of Propositions \ref{l:19cur} and \ref{l:19L^2}. Their proofs will be presented in Section \ref{s:195}.

\vskip0.24cm

\section{Proof of Proposition \ref{l:28boff}}\label{s:196}
\vskip0.24cm
In this section, we prove Proposition \ref{l:28boff}.
Let us first introduce the Littlewood-Paley decomposition. Let $\psi$ be a radial
 $C^\infty$ function such that $\psi(\xi)=1$ for $|\xi|\leq 1$, $\psi(\xi)=0$ for $|\xi|\geq 2$,
and $0\leq\psi(\xi)\leq1$ for all $\xi\in\R^d$. Define $\beta_k(\xi)=\psi(2^k\xi)-\psi(2^{k+1}\xi)$;
then $\beta_k$ is supported in $\{\xi:2^{-k-1}\leq|\xi|\leq 2^{-k+1}\}$. Choose $\tilde\beta$ to be a radial
 $C^\infty$ function such that $\tilde{\beta}(\xi)=1$ for $\frac12\leq|\xi|\leq 2$, $\tilde{\beta}$ is supported in $\{\xi:\frac14\leq|\xi|\leq4\}$,
and $0\leq\tilde{\beta}(\xi)\leq1$ for all $\xi\in\R^d$. Set $\tilde{\beta}_k(\xi)=\tilde{\beta}(2^k\xi)$; then it is easy to see $\beta_k=\tilde{\beta}_k\beta_k$. Define the convolution operators $\Lam_k$ and $\tilde{\Lam}_k$
with the Fourier multipliers $\beta_k$ and $\tilde{\beta}_k$, respectively.
 That is,
$$\widehat{{\Lam}_kf}(\xi)=\beta_k(\xi)\hat{f}(\xi),\quad \ \widehat{\tilde{\Lam}_kf}(\xi)=\tilde{\beta}_k(\xi)\hat{f}(\xi).$$
Then by the construction of $\beta_k$ and $\tilde{\beta}_k$, we have
$\Lam_k=\tilde{\Lam}_k\Lam_k$ and the identity ${\rm{I}}=\sum\limits_{k\in\mathbb{Z}}\Lam_k.$

The author in \cite[Page 1468, Line -8]{Lai24} used the Schur lemma to show that for all $k-j+n+s>0$,
\Be\label{e:28gnj}
\|\Lam_{k}g_{n-j,s}\|_{L_2(\A)}^2\lc2^{j-n-s-k}\|g_{n-j,s}\|_{L_2(\A)}^2.
\Ee

Recall   $K_j $ is the kernel of the operator $T_j$, i.e.\ $K_j (x)=\Om(x)\phi_j(x)|x|^{-d}$.  Then we have the following useful inequality
\Be\label{e:19plan}
\sum_{j\in\Z}|\widehat{K_j }(\xi)|^2\lc\C^2_\Omega
\Ee
holds for almost every $\xi\in\R^d$ (see \eg \cite[(5-2)]{Lai24} for a proof).

By inserting the Littlewood-Paley decomposition, we write
\Be\label{e:28L2}
\begin{split}
\big\|\sum_j  T_jg_{n-j,s}\big\|^2_{L_2(\A)}&=\big\|\sum_k\sum_j  T_j\Lam_{k}g_{n-j,s}\big\|^2_{L_2(\A)}\\
&\lc\sum_k\tau\int\big|\sum_j  \widehat{K_j}(\xi)\beta_{k}(\xi)\widehat{g_{n-j,s}}(\xi)\big|^2d\xi
\end{split}
\Ee
where the second inequality follows from the Plancherel  theorem and the fact that the supports of $\beta_{k}$ are essentially disjoint (in fact $\supp(\beta_{k})$ only intersects $\supp(\beta_{k+l})$ for $l=-1,0,1$).
Next we split the sum over $j$ into two cases: $j< k+(n+s)\epsilon$ and $j\geq k+(n+s)\epsilon$, where $\epsilon\in(0,1)$ is a fixed constant.

First let us consider the sum over $j<k+(n+s)\epsilon$. By the convex inequality for operator-valued functions \eqref{e:28conv} and the estimate of $K_j$ in \eqref{e:19plan}, we get
\Be\label{e:28Kj11}
\begin{split}
&\quad\sum_{k}\tau\int\big|\sum_{j<k+(n+s)\epsilon}  \widehat{K_j}(\xi)\beta_{k}(\xi)\widehat{g_{n-j,s}}(\xi)\big|^2d\xi\\
&\lc\sum_{k}\tau\int\Big(\sum_{j}|\widehat{K_j}(\xi)|^2\Big)  \Big(\sum_{j<k+(n+s)\epsilon} \big|\beta_{k}(\xi)\widehat{g_{n-j,s}}(\xi)\big|^2\Big)d\xi\\
&\lc\C_\Om^2\sum_{k}\sum_{j<k+(n+s)\epsilon}\|\check{\beta}_{k}*g_{n-j,s}\|_{L_2(\A)}^2,
\end{split}
\Ee
where in the last inequality we also use the Plancherel theorem. Notice that $j<k+(n+s)\epsilon$ implies that $k-j+n+s>0$ since $\epsilon\in(0,1)$.
Applying \eqref{e:28gnj} and Lemma \ref{l:19gbasic}, the above estimate is bounded by
\Be\label{e:28jleqk}
\begin{split}
\C_\Om^2\sum_{j}\sum_{j-k<(n+s)\epsilon} 2^{2(j-k-n-s)}\|g_{n-j,s}\|_{L_2(\A)}^2
\lc\C_\Om2^{-2(1-\epsilon)(n+s)}\lam\|f\|_{L_1(\A)}.
\end{split}
\Ee

In the following we consider the sum over $j\geq k+(n+s)\epsilon$.  By the convex inequality for operator-valued functions \eqref{e:28conv}, we obtain
\Be\label{e:28kjgeq}
\begin{split}
&\quad\sum_{k}\tau\int\big|\sum_{j\geq k+(n+s)\epsilon}  \widehat{K_j}(\xi)\beta_{k}(\xi)\widehat{g_{n-j,s}}(\xi)\big|^2d\xi\\
&\leq\sum_k\tau\int\Big(\sum_{j\geq k+(n+s)\epsilon}|\widehat{K}_j({\xi})\beta_k(\xi)|^2\Big)\Big(\sum_j  |\tilde{\beta}_{k}(\xi)\widehat{g_{n-j,s}}(\xi)|^2\Big)d\xi.
\end{split}
\Ee

Recall that $K_j(x)=\frac{\Om(x)}{|x|^d}\phi_j(x)$. For fixed $k$ and $j$, we decompose $\Om(\tet)$ into two parts: $\Om_1(\tet)$
and $1-\Om_1(\tet)$, where $\Om_1(\tet)=\Om(\tet)\chi_{\{\tet\in\S^{d-1}:|\Om(\tet)|\leq2^{(j-k)\nu}\|\Om\|_1\}}$ for some constant $\nu\in(0,1/2)$.

We first consider $\Om_1$. Thus, we give an estimate of \eqref{e:28kjgeq} with the kernel $\Omega$ replaced by $\Omega_1$. By making a change of variable $x=r\tet$, we get
\Be\label{e:19kjpl}
|\widehat{ K_j}(\xi)|\leq\int_{\S^{d-1}}|\Om_1(\tet)|\Big|\int_{\R}e^{-ir\inn{\tet}{\xi}}\phi_j(r)r^{-1}dr\Big|d\si(\tet).
\Ee

It is easy to see that $|\int_{\R}e^{-ir\inn{\tet}{\xi}}\phi_j(r)r^{-1}dr|$ is finite. By integration by parts with respect to $r$, we get
\Bes
\Big|\int_{\R}e^{-ir\inn{\tet}{\xi}}\phi_j(r)r^{-1}dr\Big|
=\Big|\int_{\R}e^{-ir\inn{\tet}{\xi}}{\inn{\tet}{\xi}}^{-1}\pari_r[\phi_j(r)r^{-1}]dr\Big|\lc(2^j|\xi|)^{-1}|\inn{\tet}{\xi'}|^{-1},
\Ees
where $\xi'=\xi/|\xi|$. Interpolating these two estimates we get that for any $\del\in(1/2,1)$,
\Bes
\Big|\int_{\R}e^{-ir\inn{\tet}{\xi}}\phi_j(r)r^{-1}dr\Big|\lc(2^j|\xi|)^{-\del}|\inn{\tet}{\xi'}|^{-\del}.
\Ees
Plugging the above estimate into \eqref{e:19kjpl} and using the fact $\int_{\S^{d-1}}|\inn{\tet}{\xi'}|^{-\del}d\si(\tet)<\infty$, we obtain
\Be
\sum_{j\geq k+(n+s)\epsilon}|\widehat{K_j}(\xi)\beta_k(\xi)|^2\lc\sum_{j\geq k+(n+s)\epsilon}(2^{j-k})^{-2\del+2\nu}\|\Om\|^2_1
\lc2^{-(n+s)\epsilon(2\del-2\nu)}\|\Om\|^2_1.
\Ee
Combining the above estimate, in the case of $\Om_1$ we hence get
\Be\label{e:28Kj12}
\begin{split}
\eqref{e:28kjgeq}&\lc2^{-(n+s)\eps(2\del-2\nu)}\|\Om\|_1^2\sum_{k\in\Z}\tau\int\sum_j  \big|\tilde{\beta}_k(\xi)\widehat{g_{n-j,s}}(\xi)\big|^2d\xi\\
&\lc2^{-(n+s)\eps(2\del-2\nu)}\|\Om\|_1^2 \sum_j\|g_{n-j,s}\|_{L_2(\A)}^2
\lc2^{-(n+s)\eps(2\del-2\nu)}\C_\Om\lam\|f\|_{L_1(\A)},
\end{split}
\Ee
where in the second inequality we apply the supports of $\tilde\beta_{k}$ are essentially disjoint, the Plancherel theorem, and the last inequality follows from Lemma \ref{l:19gbasic}.

Next we consider the other case $\Om_2$. In this case, we have
\Be
|\widehat{K}_j(\xi)|\lc\int_{\{\theta\in\S^{d-1}:|\Om(\tet)|>2^{(j-k)\nu}\|\Om\|_1\}}|\Om(\tet)|d\sigma(\tet).
\Ee
Hence $\eqref{e:28kjgeq}$ is majorized by
\Be\label{e:28Kj2}
\begin{split}
\sum_k\tau\int&\Big(\sum_{l\geq(n+s)\epsilon}\Big|\int_{\{\theta\in\S^{d-1}:|\Om(\tet)|>2^{l\nu}\|\Om\|_1\}}|\Om(\tet)|d\sigma(\tet)\Big|^2\Big)\Big(\sum_j  |\tilde{\beta}_{k}(\xi)\widehat{g_{n-j,s}}(\xi)|^2\Big)d\xi\\
&\lc \sum_{l\geq(n+s)\epsilon}\Big(\int_{\{\theta\in\S^{d-1}:|\Om(\tet)|>2^{l\nu}\|\Om\|_1\}}|\Om(\tet)|d\sigma(\tet)\Big)^2 \sum_{j\in\Z}\|g_{n-j,s}\|_{L_2(\A)}^2\\
&\lc\sum_{l\geq(n+s)\epsilon}\Big(\int_{\{\theta\in\S^{d-1}:|\Om(\tet)|>2^{l\nu}\|\Om\|_1\}}|\Om(\tet)|d\sigma(\tet)\Big)^2\C_\Om^{-1}\lam\|f\|_{L_1(\A)},
\end{split}
\Ee
where in the first inequality we also use the Plancherel theorem, and the second inequality follows from Lemma \ref{l:19gbasic}.
For convenience, set
$$B(s,n)=\sum_{l\geq(n+s)\epsilon}\Big(\int_{\{\theta\in\S^{d-1}:|\Om(\tet)|>2^{l\nu}\|\Om\|_1\}}|\Om(\tet)|d\sigma(\tet)\Big)^2\C_\Om^{-1}.$$

Now by \eqref{e:28L2}, \eqref{e:28Kj11}, \eqref{e:28jleqk}, \eqref{e:28Kj12} and \eqref{e:28Kj2}, we conclude that
\Bes
\big\|\sum_j  T_jg_{n-j,s}\big\|^2_{L_2(\A)}\lc D(s,n)\lam\|f\|_{L_1(\A)}
\Ees
where the constant $D(s,n)$ is given by
\Bes
D(s,n)=2^{-2(1-\epsilon)(n+s)}\C_\Om+2^{-(n+s)(2\del-2\nu)}\C_\Om+B(s,n).
\Ees

To finish the proof of Proposition \ref{l:28boff}, it remains to show that the constant $D(s,n)$ satisfies the decay estimate \eqref{e:19boffdecay}. It is easy to see that the first and second factors in $D(s,n)$ satisfy \eqref{e:19boffdecay}. In the following, we prove that $B(s,n)$ satisfies \eqref{e:19boffdecay}. By an elementary calculation, we get
\Bes
\begin{split}
\sum_{n\geq 100}\sum_{s\geq1}&\Big(B(s,n)\Big)^{\frac12}\leq\sum_{n\geq 100}\sum_{s\geq1}\Big[\int_{\alpha\in\S^{d-1}}\int_{\tet\in\S^{d-1}}|\Om(\tet)||\Om(\alpha)|\times\\
&\#\Big\{l:\, (n+s)\epsilon\leq l <\frac{1}{\nu}
\min\{\log^+\frac{|\Om(\tet)|}{\|\Om\|_{1}},\log^+\frac{|\Om(\alpha)|}{\|\Om\|_{1}}\} \Big\}d\si(\tet)d\si(\alpha)\Big]^{\frac12}\C_\Om^{-\frac12}.
\end{split}
\Ees
Applying the H\"older inequality, the above estimate is bounded by
\Bes
\begin{split}
&\quad\C_\Om^{-\frac12}\Big[\sum_{n\geq 100}\sum_{s\geq1}\int_{\tet\in\S^{d-1}}\#\Big\{l:\, (n+s)\epsilon\leq l <\frac{1}{\nu}
\log^+\frac{|\Om(\tet)|}{\|\Om\|_{1}} \Big\}^{\frac12}|\Om(\tet)|d\si(\tet)\Big]^{\frac12}\\
&\quad\quad\times\Big[\sum_{n\geq 100}\sum_{s\geq1}\int_{\alpha\in\S^{d-1}}\#\Big\{l:\, (n+s)\epsilon\leq l <\frac{1}{\nu}
\log^+\frac{|\Om(\alpha)|}{\|\Om\|_{1}} \Big\}^{\frac12}|\Om(\alpha)|d\si(\alpha)\Big]^{\frac12}\\
&\leq\C_\Om^{-\frac12}\int_{\tet\in\S^{d-1}}
\#\Big\{n:\, 100\leq n <\frac{1}{\epsilon \nu}
\log^+\frac{|\Om(\tet)|}{\|\Om\|_{1}} \Big\}
\#\Big\{s:\, 1 \leq s <\frac{1}{\epsilon \nu}
\log^+\frac{|\Om(\tet)|}{\|\Om\|_{1}} \Big\}\\
&\quad\quad\times\#\Big\{l:\, 1\leq l <\frac{1}{\nu}
\log^+\frac{|\Om(\tet)|}{\|\Om\|_{1}} \Big\}^{\frac12}|\Om(\tet)|d\si(\tet)\\
&\lc\C_\Om^{-\frac12}\int_{\tet\in\S^{d-1}}\Big(1+\log^+\frac{|\Om(\tet)|}{\|\Om\|_{1}} \Big)^{\frac52}|\Om(\tet)|d\si(\tet)\lc\C_\Om^{\frac12},
\end{split}
\Ees
which is the desired estimate. Hence we have completed the proof.
$\hfill{} \Box$
\vskip0.24cm

\section{Proofs of Propositions \ref{l:19cur} and \ref{l:19L^2}}\label{s:195}
\vskip0.24cm

\subsection{Proof of Proposition \ref{l:19cur}}\label{s:1951}\quad
\vskip0.24cm
Denote the kernel of the operator $T_{j,\iota}^{n,s}$ by $
K_{j,\iota}^{n,s}(x-y)$. Then it is easy to see that
$$\big\|K_{j,\iota}^{n,s}\big\|_{L_1(\R^d)}\lc\int_{D^{\iota}}\int_{2^{j-1}}^{2^{j+1}}|\Om(\theta)|r^{d-1}2^{-jd}dr d\si(\theta)
\lc\int_{D^\iota}|\Om(\theta)|d\si(\theta).$$
Therefore by the Chebyshev inequality, the triangle inequality, and  the H\"older inequality to remove the projection $\zet$, we get
\Bes
\begin{split}
\vphi \big(\big|\zet\sum_{n\geq100}\sum_{s\geq0}\sum_{j\in
\mathbb{Z}}T_{j,\iota}^{n,s}b_{n-j,s}  \zet\big|>\lambda\big)&\leq\lam^{-1}\sum_{n\geq100}\sum_{s\geq0}\sum_{j\in
\mathbb{Z}}\|T_{j,\iota}^{n,s}b_{n-j,s}  \|_{L_1(\A )}\\
&\leq{\lam}^{-1}\sum_{n\geq100}\sum_{s\geq0}\sum_{j\in
\mathbb{Z}}\|K_{j,\iota}^{n,s}\|_{L_1(\R^d)}\|b_{n-j,s}\|_{L_1(\A)}\\
&\lc{\lam}^{-1}\sum_{n\geq100}\sum_{s\geq0}\int_{D^{\iota}}|\Om(\theta)|d\si(\theta)\sum_{j}\|b_{n-j,s}\|_{L_1(\A)}.
\end{split}
\Ees
Now applying property (i) of Lemma \ref{l:19bbasic}, the above estimate is bounded by
\Bes
\begin{split}
&{\lam}^{-1}\|f\|_{L_1(\A)}\int_{\S^{d-1}}\#\big\{(n,s):\, n\geq 100, s\geq0, 2^{\iota (n+s)}\leq \frac{|\Om(\tet)|}{\|\Om\|_{1}}\big\}|\Om(\tet)|d\si(\tet)\\
&\lc{\lam}^{-1}\|f\|_{L_1(\A)}\int_{\S^{d-1}}|\Om(\tet)|\big(1+\log^+(|\Om(\tet)|/\|\Om\|_{1})\big)^2 d\si(\tet)\lc{\lam}^{-1}\C_\Om\|f\|_{L_1(\A)},
\end{split}
\Ees
which completes the proof.
$\hfill{} \Box$
\vskip0.24cm

\subsection {Proof of Proposition \ref{l:19L^2}}\label{s:1952}\quad
\vskip0.24cm
First we have the following observation about the orthogonality of the support of $\mathcal{F}(G_{k,v})$: For a fixed $k\geq 100$, we get \begin{equation}\label{e:19obser}
\sup\limits_{\xi\neq0}\sum\limits_{v\in\Theta_{k}}|\Phi^2(2^{k\ga}\inn{e^k_v}{\xi/|\xi|})|\lc2^{k\ga(d-2)}.
\end{equation}
In fact, by homogeneity of $\Phi^2(2^{k\ga}\inn{e^k_v}{\xi/|\xi|})$, it suffices to take the supremum over the sphere $\S^{d-1}$. For $|\xi|=1$ and $\xi\in\supp\ \Phi^2(2^{k\ga}\inn{e^k_v}{\xi/|\xi|})$, denote by $\xi^{\bot}$ the hyperplane perpendicular to $\xi$. Then it is easy to see that
\begin{equation}\label{e:19e^n_v}
\text{dist}(e^k_v,\xi^\bot)\lc2^{-k\ga}.
\end{equation}
Since the mutual distance of the $e^k_v$'s is bounded below by $2^{-k\ga-4}$, there are at most $2^{k\ga(d-2)}$ vectors satisfying
(\ref{e:19e^n_v}). This yields (\ref{e:19obser}).

Applying the Plancherel theorem, the convex inequality for operator-valued functions \eqref{e:28conv}, the estimate \eqref{e:19obser} and finally the Plancherel theorem again, we get
\begin{equation}\label{e:19L^2key}
\begin{split}
&\quad\Big\|\sum\limits_{v\in\Theta_{n+s}}\sum_jG_{n+s,v}T^{n,s,v}_jb_{n-j,s}\Big\|^2_{L_2(\A)}\\
&=(2\pi)^{-\frac{d}{2}}\int_{\R^d}
\tau\Big(\Big|\sum\limits_{v\in\Theta_{n+s}}\Phi(2^{(n+s)\ga}\inn{e^{n+s}_v}{\xi/|\xi|})
\sum_j\mathcal{F}\big(T^{n,s,v}_jb_{n-j,s}\big)(\xi)\Big|^2\Big)d\xi\\
&\lc\int_{\R^d}\sum\limits_{v\in\Theta_{n+s}}\Phi^2(2^{(n+s)\ga}\inn{e^{n+s}_v}{\xi/|\xi|})
 \tau\Big(\sum\limits_{v\in\Theta_{n+s}}\Big|\sum_j\mathcal{F}\big(T^{n,s,v}_jb_{n-j,s}\big)(\xi)\Big|^2\Big)d\xi\\
&\lc2^{(n+s)\ga(d-2)} \sum\limits_{v\in\Theta_{n+s}}\big\|\sum_jT^{n,s,v}_jb_{n-j,s}\big\|^2_{L_2(\A)}.
\end{split}
\end{equation}

We claim that for a fixed $e^{n+s}_v$, the following estimate holds
\begin{equation}\label{e:19L^2}
\Big\|\sum_j T^{n,s,v}_jb_{n-j,s}\Big\|^2_{L_2(\A)}\lc(s+1)2^{-2(n+s)\ga(d-1)+2(n+s)\iota}\lam\|\Om\|_{1}\|f\|_{L_1(\A)}.
\end{equation}
Then, since $\card(\Theta_{n+s})\lc2^{(n+s)\ga(d-1)}$, applying (\ref{e:19L^2key}) and (\ref{e:19L^2}) we get
\begin{equation*}
\begin{split}
\Big\|\sum\limits_{v\in\Theta_{n+s}}\sum\limits_jG_{n+s,v}T^{n,s,v}_jb_{n-j,s}\Big\|^2_{L_2(\A)}\lc(s+1)2^{-(n+s)\ga+2(n+s)\iota}\lam\C_\Om\|f\|_{L_1(\A)},
\end{split}
\end{equation*}
which is the asserted bound of Proposition \ref{l:19L^2}. Thus, to finish the proof of Proposition \ref{l:19L^2}, it is sufficient to prove (\ref{e:19L^2}).

Recall the definition of $b_{n-j,s}$ in \eqref{e:19db}.
By the triangle inequality, to prove \eqref{e:19L^2} it is enough to prove the following lemma.
\begin{lemma}\label{L:28fnjs}
With all the above notation, we have
  \Bes
\begin{split}
\Big\|\sum_j T^{n,s,v}_jp_{n-j}f_{n-j+s}p_{n-j+s}\Big\|^2_{L_2(\A)}
&\lc (s+1)2^{-2(n+s)\ga(d-1)+2(n+s)\iota}\lam\|\Om\|_{1}\|f\|_{L_1(\A)};\\
\Big\|\sum_j T^{n,s,v}_jp_{n-j+s}f_{n-j+s}p_{n-j}\Big\|^2_{L_2(\A)}
&\lc (s+1)2^{-2(n+s)\ga(d-1)+2(n+s)\iota}\lam\|\Om\|_{1}\|f\|_{L_1(\A)};\\
\Big\|\sum_j T^{n,s,v}_jp_{n-j}fp_{n-j+s}\Big\|^2_{L_2(\A)}&\lc (s+1)2^{-2(n+s)\ga(d-1)+2(n+s)\iota}\lam\|\Om\|_{1}\|f\|_{L_1(\A)};\\
\Big\|\sum_j T^{n,s,v}_jp_{n-j+s}fp_{n-j}\Big\|^2_{L_2(\A)}&\lc (s+1)2^{-2(n+s)\ga(d-1)+2(n+s)\iota}\lam\|\Om\|_{1}\|f\|_{L_1(\A)}.
\end{split}
\Ees
\end{lemma}

\begin{proof}
We will only give the detailed proof of the third inequality, since the other three terms can be handled similarly. Indeed, the fourth inequality is symmetric to the third one, so the proof is the same. For the first and second inequalities, we consider the conditional expectation $f_{n-j+s}$, which is easier to deal with than the function appearing in the third one. One can easily adapt the proof of the third inequality to the setting of the first and second inequalities, and we will point out how to modify the arguments at the end of the proof.

Write $$\Big\|\sum_j T^{n,s,v}_jp_{n-j}fp_{n-j+s}\Big\|^2_{L_2(\A)}=:I+II$$ where
\Bes
\begin{split}
I=:\sum_{j\in\Z}\sum_{|i-j|> s}\tau\int\big(T^{n,s,v}_jp_{n-j} f p_{n-j+s}(x)\big)^*
T^{n,s,v}_ip_{n-i} f p_{n-i+s}(x)dx
\end{split}
\Ees
and
\Bes
\begin{split}
II=:\sum_{j\in\Z}\sum_{|i-j|\leq s}\tau\int\big(T^{n,s,v}_jp_{n-j} f p_{n-j+s}(x)\big)^*
T^{n,s,v}_ip_{n-i} f p_{n-i+s}(x)dx.
\end{split}
\Ees
It should be pointed out that the split above has not appeared before. In fact it is not necessary in the classical arguments (see \eg \cite{See96}) but is a key step in the noncommutative setting. Therefore, it suffices to establish the required estimates for the terms $I$ and $II$, respectively.
\vskip0.24cm
\textbf{Proof of the term $I$}. Our main strategy for this term is the $TT^*$ method. Rewrite $I$ as the sum of
\Be\label{e:28ijgeq21}
\sum_{j\in\Z}\sum_{i<j-s}\tau\int\big(T^{n,s,v}_jp_{n-j} f p_{n-j+s}(x)\big)^* T^{n,s,v}_ip_{n-i} f p_{n-i+s}(x)dx
\Ee
and
\Be\label{e:28ijgeqs22}
\sum_{i\in\Z}\sum_{j<i-s}\tau\int T^{n,s,v}_ip_{n-i} f p_{n-i+s}(x)\big(T^{n,s,v}_jp_{n-j} f p_{n-j+s}(x)\big)^*dx
\Ee
where we use the trace property $\tau(ab)=\tau(ba)$ to obtain the second term.
In the following we only give an estimate for \eqref{e:28ijgeq21}, since \eqref{e:28ijgeqs22} can be dealt with in a similar manner.

Set the kernel of $T_j^{n,s,v}$ as
$
K_j^{n,s,v}(x)=\Ga^{n+s}_v(x)\Om(x)\phi_j(x)|x|^{-d}.
$
Define $\tilde{K}_j^{n,s,v}(x)=K_j^{n,s,v}(-x)$. Then, applying the Fubini theorem to change the order of integration, \eqref{e:28ijgeq21} equals
\Bes
\begin{split}
\sum_j\tau\int{p_{n-j+s} f p_{n-j}}(x)\Big[\sum_{i< j-s}(\tilde{K}_j^{n,s,v})*K_i^{n,s,v}*{(p_{n-i} f p_{n-i+s})}(x)\Big]dx.
\end{split}
\Ees
By  the H\"older inequality, this is majorized by
\Bes
\begin{split}
\sum_j\|{p_{n-j+s} f p_{n-j}}\|_{L_1(\A)}\Big\|\sum_{i< j-s}(\tilde{K}_j^{n,s,v})*K_i^{n,s,v}*{(p_{n-i} f p_{n-i+s})}\Big\|_{\A}.
\end{split}
\Ees

Before proceeding with the proof further, we first claim that
\Be\label{e:28sumjfnjs}
\sum_j\|{p_{n-j+s} f p_{n-j}}\|_{L_1(\A)}\lc\|f\|_{L_1(\A)},
\Ee
and for any $j\in\Z$,
\Be\label{e:28sumijs}
\Big\|\sum_{i< j-s}\tilde{K}_j^{n,s,v}*K_i^{n,s,v}*{(p_{n-i} f p_{n-i+s})}\Big\|_{\A}\lc2^{-2(n+s)\ga(d-1)+2(n+s)\iota}\lam\|\Om\|_{1}.
\Ee
By the above two claims \eqref{e:28sumjfnjs} and \eqref{e:28sumijs}, we immediately obtain the desired bound for $I$.

Let us now prove \eqref{e:28sumjfnjs}. By  the H\"older inequality,
\Be\label{e:28pnjin}
\begin{split}
\sum_j\|{p_{n-j+s} f p_{n-j}}\|_{L_1(\A)}&\leq\Big(\sum_j\|p_{n-j+s} f ^{1/2}\|_{L_2(\A)}^2\Big)^{\frac12}
\Big(\sum_j\| f ^{1/2}p_{n-j}\|_{L_2(\A)}^2\Big)^{\frac12}\\
&=\Big(\sum_j\vphi(p_{n-j+s}f)\Big)^{\frac12}\Big(\sum_j\vphi(p_{n-j}f)\Big)^{\frac12}\leq\|f\|_{L^1(\A)},
\end{split}
\Ee
which is exactly the estimate \eqref{e:28sumjfnjs}. We also point out that by the H\"older inequality and the trace invariance of conditional expectations,
\Be\label{e:28pnjinaverage}
\begin{split}
\sum_j\|{p_{n-j+s} f_{n-j+s} p_{n-j}}\|_{L_1(\A)}\leq\|f\|_{L^1(\A)}.
\end{split}
\Ee

Now we turn to the estimate in \eqref{e:28sumijs}. Fix $j\in\Z$. By the Hilbert-valued H\"older inequality in Lemma \ref{l:28holder}, the left-hand side of \eqref{e:28sumijs} is majorized by
\Be\label{e:28sumijs2}
\begin{split}
\Big\|\Big(\sum_{i< j-s}| \tilde{K}_j^{n,s,v} &*K_i^{n,s,v}|*{(p_{n-i+s} f p_{n-i+s})}\Big)^{\frac12}\Big\|_{L_\infty(\A)}\\
&\quad\times\Big\|\Big(\sum_{i< j-s}| \tilde{K}_j^{n,s,v} *K_i^{n,s,v}|*{(p_{n-i} f p_{n-i})}\Big)^{\frac12}\Big\|_{L_\infty(\A)}.
\end{split}
\Ee
For convenience, set
\Bes
S_{j,i,1}(x) =| \tilde{K}_j^{n,s,v} *K_i^{n,s,v}|*{(p_{n-i+s} f p_{n-i+s})},
\Ees
and
\Bes
S_{j,i,2}(x) =| \tilde{K}_j^{n,s,v} *K_i^{n,s,v}|*{(p_{n-i} f p_{n-i})}.
\Ees

We first recall the following noncommutative Calder\'on--Zygmund property
\Be\label{e:28PQ}
p_{n-i+s} f_{n-i+s} p_{n-i+s}\lc\lambda\C_\Om^{-1}p_{n-i+s}\quad\text{or}\quad p_Qf_Qp_Q\lc\lambda\C_\Om^{-1}p_{Q}
\Ee
for any $Q\in\mathcal{Q}_k$, which is a consequence of property (ii) in Lemma \ref{l:28cucu}.
Denote the support of $K_j^{n,s,v}$ and $\tilde{K}_j^{n,s,v}*K_i^{n,s,v}$ by $E^{n,s,v}_j$ and $E^{n,s,v}_{j,i}$, respectively. Then it is easy to see that
\Bes
\begin{split}
E_j^{n,s,v}&\subset\{x\in\R^d:|\frac{x}{|x|}-e^{n+s}_v|\leq2^{(n+s)\ga}, 2^{j-1}\leq|x|\leq2^{j+1}\}\\
&\subset\{x\in \R^d:|\inn{x}{e^{n+s}_v}|\leq 2^{j+1},|x-\inn{x}{e^{n+s}_v}e^{n+s}_v|\leq 2^{j+1-(n+s)\ga}\},
\end{split}
\Ees
and for $i<j-s$,
\Bes
\begin{split}
E_{j,i}^{n,s,v}
&\subset E_{j,\sharp}^{n,s,v}:=\big\{x\in \R^d:|\inn{x}{e^{n+s}_v}|\leq 2^{j+2},|x-\inn{x}{e^{n+s}_v}e^{n+s}_v|\leq 2^{j+2-(n+s)\ga}\big\}.
\end{split}
\Ees
Notice that $E_{j,\sharp}^{n,s,v}$ is independent of $i$.

By an elementary calculation, we obtain the following estimate for the kernel $\tilde{K}_j^{n,s,v}*K_i^{n,s,v}$:
\Be\label{e:28kerneltt*}
|\tilde{K}_j^{n,s,v}*K_i^{n,s,v}(x)|\lc2^{-jd}2^{(2\iota-\ga(d-1)) (n+s)}\|\Om\|_1^2\chi_{E^{n,s,v}_{j,\sharp}}(x).
\Ee

\medskip

Next we return to $S_{j,i,1}(x)$ and $S_{j,i,2}(x)$.
Applying \eqref{e:28kerneltt*}, we obtain an estimate for $S_{j,i,1}(x)$ as follows
\Be\label{e:28sji1}
\begin{split}
\sum_{i< j-s}S_{j,i,1}(x)&\lc2^{-jd}2^{(2\iota-\ga(d-1)) (n+s)}\lam\|\Om\|_1^2\sum_{i< j-s}\chi_{{E}^{n,s,v}_{j,\sharp}}*p_{n-i+s}fp_{n-i+s}(x).
\end{split}
\Ee

Using the definition of $p_{n-i+s}$, we get
\Be\label{e:28pnifnisnos}
\begin{split}
\sum_{i< j-s}\chi_{{E}^{n,s,v}_{j,\sharp}}*p_{n-i+s} f p_{n-i+s}(x)
&=\sum_{i< j-s}\sum_{Q\in\mathcal{Q}_{n-i+s}}\int_{y\in(x-E_{j,\sharp}^{n,s,v})\cap Q}p_{Q} f (y)p_{Q}dy\\
&\leq \sum_{i< j-s}\sum_{Q\in\mathcal{Q}_{n-i+s}\atop (x-E_{j,\sharp}^{n,s,v})\cap Q\neq\emptyset }p_Qf_Qp_Q|Q|.
\end{split}
\Ee
Applying the noncommutative Calder\'on--Zygmund property \eqref{e:28PQ}, the above estimate is bounded by
\Bes
\begin{split}
\lam\C_\Om^{-1}\sum_{i< j-s}\sum_{Q\in\mathcal{Q}_{n-i+s}\atop (x-E_{j,\sharp}^{n,s,v})\cap Q\neq\emptyset }\int_Qp_{n-i+s}(y)dy\lc\lam\C_\Om^{-1}\sum_{i< j-s}\int_{\textbf{E}^{n,s,v}_{j,i,\sharp}(x)}p_{n-i+s}(y)dy
\end{split}
\Ees
where
\Bes
\textbf{E}^{n,s,v}_{j,i,\sharp}(x)=\bigcup_{Q\in\mathcal{Q}_{n-i+s}\atop (x-E_{j,\sharp}^{n,s,v})\cap Q\neq\emptyset } Q.
\Ees

Notice that $E_{j,\sharp}^{n,s,v}$ is supported in a rectangle with one sidelength at most $2^{j+2}$ and $d-1$ sidelengths at most $2^{j+2-(n+s)\gamma}$.
For any $Q\in\mathcal{Q}_{n-i+s}$, the sidelength $l(Q)=2^{i-n-s}<2^{j+2-(n+s)\gamma}$. Since all cubes in $\mathcal{Q}_{n-i+s}$ are disjoint, we see that ${E}^{n,s,v}_{j,i,\sharp}(x)$ is supported in a rectangle with one sidelength at most $2^{j+3}$ and $d-1$ sidelengths at most $2^{j+3-(n+s)\gamma}$, and we denote this rectangle by $\textbf{E}^{n,s,v}_{j,\sharp}(x)$ since it is independent of $i$. It is easy to see that the measure of $\textbf{E}^{n,s,v}_{j,\sharp}(x)$ satisfies
\Be\label{e:28ensvj}|\textbf{E}^{n,s,v}_{j,\sharp}(x)|\lc2^{jd-(n+s)\ga(d-1)}.
\Ee
Combining the above arguments,
we obtain that \eqref{e:28pnifnisnos} is bounded by
\Be\label{e:28Ejxnos}
\begin{split}
\lam\C_\Om^{-1}\sum_{i< j-s}\int_{\textbf{E}^{n,s,v}_{j,\sharp}(x)}p_{n-i+s}(y)dy
&\lc\lam\C_\Om^{-1}\int_{\textbf{E}^{n,s,v}_{j,\sharp}(x)}\sum_{i< j-s}p_{n-i+s}(y)dy\\
&\lc\lam\C_\Om^{-1}\big|\textbf{E}^{n,s,v}_{j,\sharp}(x)\big|\lc\lam\C_\Om^{-1}2^{jd-(n+s)\ga(d-1)}.
\end{split}
\Ee

Plugging \eqref{e:28pnifnisnos} with \eqref{e:28Ejxnos} into  \eqref{e:28sji1}, we get
\Be\label{e:28sji1f}
\begin{split}
\sum_{i< j-s}S_{j,i,1}(x)\lc2^{(2\iota-2\gamma(d-1)) (n+s)}\lam\|\Om\|_1.
\end{split}
\Ee

\medskip

Concerning the term $S_{j,i,2}(x)$, the proof is slightly more complicated. By the kernel estimate \eqref{e:28kerneltt*}, we get
\Be\label{e:28sji2}
\begin{split}
\sum_{i< j-s}S_{j,i,2}(x)&\lc2^{-jd}2^{(2\iota-\ga(d-1)) (n+s)}\lam\|\Om\|_1\sum_{i< j-s}\chi_{{E}^{n,s,v}_{j,\sharp}}*p_{n-i} f p_{n-i}(x).
\end{split}
\Ee
By the definition of $p_{n-i}$, we have
\Be\label{e:28pnifnis}
\begin{split}
\sum_{i< j-s}\chi_{{E}^{n,s,v}_{j,\sharp}}*p_{n-i} f p_{n-i}(x)
&=\sum_{i< j-s}\sum_{Q\in\mathcal{Q}_{n-i}}\int_{y\in(x-E_{j,\sharp}^{n,s,v})\cap Q}p_{Q} f (y)p_{Q}dy\\
&\leq \sum_{i< j-s}\sum_{Q\in\mathcal{Q}_{n-i}\atop (x-E_{j,\sharp}^{n,s,{v}})\cap Q\neq\emptyset }p_Qf_Qp_Q|Q|.
\end{split}
\Ee
Applying property \eqref{e:28PQ}, the above estimate is bounded by
\Bes
\lam\C_\Om^{-1}\sum_{i< j-s}\sum_{Q\in\mathcal{Q}_{n-i}\atop (x-E_{j,\sharp}^{n,s,v})\cap Q\neq\emptyset }\int_Qp_{n-i}(y)dy\lc\lam\C_\Om^{-1}\sum_{i< j-s}\int_{\tilde{\textbf{E}}^{n,s,v}_{j,i,\sharp}(x)}p_{n-i}(y)dy
\Ees
where
\Bes
\tilde{\mathbf{E}}^{n,s,v}_{j,i,\sharp}(x)=\bigcup_{Q\in\mathcal{Q}_{n-i}\atop (x-E_{j,\sharp}^{n,s,v})\cap Q\neq\emptyset } Q.
\Ees
For any $Q\in\mathcal{Q}_{n-i}$ with $i<j-s$, the sidelength $l(Q)=2^{i-n}<2^{j-n-s}<2^{j+2-(n+s)\gamma}$. Note that this is the point where we make use of the crucial condition $i<j-s$. Since all cubes in $\mathcal{Q}_{n-i}$ are disjoint, $\tilde{\mathbf{E}}^{n,s,v}_{j,i,\sharp}(x)$ is supported in a rectangle with one sidelength at most $2^{j+3}$ and $d-1$ sidelengths at most $2^{j+3-(n+s)\gamma}$, and we denote this rectangle by $\tilde{\mathbf{E}}^{n,s,v}_{j,\sharp}(x)$ since it is independent of $i$. Its measure satisfies
\Bes
|\tilde{\mathbf{E}}^{n,s,v}_{j,\sharp}(x)|\lc2^{jd-(n+s)\ga(d-1)}.
\Ees
Combining all the above arguments, the estimate for \eqref{e:28pnifnis} is bounded by
\Be\label{e:28Ejx}
\begin{split}
\lam\C_\Om^{-1}\sum_{i< j-s}\int_{\tilde{\mathbf{E}}^{n,s,v}_{j, \sharp}(x)}p_{n-i}(y)dy
&\lc\lam\C_\Om^{-1}\int_{\tilde{\mathbf{E}}^{n,s,v}_{j,\sharp}(x)}\sum_{i< j-s}p_{n-i}(y)dy\\
&\lc\lam\C_\Om^{-1}\big|\tilde{\mathbf{E}}^{n,s,v}_{j,\sharp}(x)\big|
\lc\lam\C_\Om^{-1}2^{jd-(n+s)\ga(d-1)}.
\end{split}
\Ee

In view of \eqref{e:28sji2} and \eqref{e:28pnifnis} with the estimate \eqref{e:28Ejx}, we get
\Be\label{e:28sji2f}
\begin{split}
\sum_{i< j-s}S_{j,i,2}(x)\lc2^{(2\iota-2\gamma(d-1)) (n+s)}\lam\|\Om\|_1.
\end{split}
\Ee

Now combining the estimates \eqref{e:28sji1f}, \eqref{e:28sji2f} and \eqref{e:28sumijs2}, we finish the proof of \eqref{e:28sumijs}. Thus, we have established the desired bound for the term $I$.

\medskip

\textbf{Proof of the term $II$.} By  the H\"older inequality, we get
\Be\label{e:28II}
\begin{split}
II&\leq\sum_{j}\sum_{|i-j|\leq s}\Big\{\Big\| T^{n,s,v}_jp_{n-j} f p_{n-j+s}\Big\|^2_{L_2(\A)}
+\Big\|T^{n,s,v}_ip_{n-i} f p_{n-i+s}\Big\|^2_{L_2(\A)}\Big\}\\
&\lc (s+1)\sum_{j} \Big\| T^{n,s,v}_jp_{n-j} f p_{n-j+s}\Big\|^2_{L_2(\A)}.
\end{split}
\Ee

Our proof below is slightly different from the one used for the term $I$. The previous method can not be applied here because the geometric argument appearing in $\tilde{\textbf{E}}_{j,\sharp}^{n,s,v}(x)$ may not work since the crucial condition $i<j-s$ is necessary. Nevertheless, the term
 $$\big\| T^{n,s,v}_jp_{n-j} f p_{n-j+s}\big\|^2_{L_2(\A)}$$
is simpler than the one in the term $I$, we can give a  direct estimate, though laborious.

By decomposing the kernel $K_j^{n,s,v}$ into its positive and negative parts, we may assume that $K_j^{n,s,v}$ is positive.

By the definition of $p_k$, we may write
\Bes
\begin{split}
T_j^{n,s,v}(p_{n-j}fp_{n-j+s})(x)&=\int_{\R^d}K_j^{n,s,v}(x-y)(p_{n-j}fp_{n-j+s})(y)dy\\
&=\sum_{Q\in\Q_{n-j+s}\atop Q\cap\{x-E_j^{n,s,v}\}\neq\emptyset}p_{{n-j}}\chi_Q\Big(\int_{Q}K_j^{n,s,v}(x-y)f(y)dy\Big)p_Q\\
&=\sum_{Q\in\Q_{n-j+s}\atop Q\cap\{x-E_j^{n,s,v}\}\neq\emptyset}\int_{Q}\Big[p_{n-j}\big(K_j^{n,s,v}(x-\cdot)f(\cdot)\big)_{n-j+s}p_{n-j+s}\Big](z)dz\\
&=\int_{\mathbf{E}_j^{n,s,v}(x)}\big[p_{n-j}\big(K_j^{n,s{v}}(x-\cdot)f(\cdot)\big)_{n-j+s}p_{n-j+s}\big](z)dz,
\end{split}
\Ees
where we use the notation
$$
\mathbf{E}_j^{n,s,v}(x)=\bigcup_{Q\in\Q_{n-j+s}\atop Q\cap\{x-E_j^{n,s,v}\}\neq\emptyset}Q.
$$
By a similar argument as done for $\mathbf{E}_{j,\sharp}^{n,s,v}(x)$ (see \eqref{e:28ensvj}), we see that $\mathbf{E}^{n,s,v}_{j}(x)$ is supported in a rectangle with one sidelength at most $2^{j+2}$ and $d-1$ sidelengths at most $2^{j+2-(n+s)\gamma}$. Hence,
\Bes
|\mathbf{E}_j^{n,s,v}(x)|\lc2^{jd-{(n+s)}\ga(d-1)}.
\Ees

Next, by the convexity inequality for operator-valued functions \eqref{e:28conv} and the preceding measure estimate of $\mathbf{E}_j^{n,s,v}(x)$, we get
\Bes
\begin{split}
\big|T_j^{n,s,v}(&p_{n-j}fp_{n-j+s})(x)\big|^2\\
&\lc2^{jd-{(n+s)}\ga(d-1)}\int_{\mathbf{E}_j^{n,s,v}(x)}\Big|p_{n-j}\Big(K_j^{n,s,v}(x-\cdot)f(\cdot)\Big)_{n-j+s}p_{n-j+s}(z)\Big|^2dz.
\end{split}
\Ees

Combining the above estimate, we obtain
\Be\label{e:19bl2fn}
\begin{split}
&\|T_j^{n,s,v}p_{n-j}fp_{n-j+s}\|^2_{L_2(\A)}\lc2^{jd-{(n+s)}\ga(d-1)}\\
&\quad\times\int_{\R^d}\int_{\mathbf{E}_j^{n,s,v}(x)}\tau\Big(\Big|p_{n-j}\Big(K_j^{n,s,v}(x-\cdot)f(\cdot)\Big)_{n-j+s}p_{n-j+s}(z)\Big|^2\Big)dzdx.
\end{split}
\Ee

Since $K_j^{n,s,v}$ is a positive function and $f$ is a positive operator-valued function in $\A$, we see that $K(x-\cdot)f(\cdot)$ is positive in $\A$. Therefore,
\Be\label{e:28Kjaverage}
\begin{split}
\big(K_j^{n,s,v}(x-\cdot)f(\cdot)\big)_{n-j+s}&=\sum_{Q\in\Q_{n-j+s}}\frac1{|Q|}\int_{Q}K_j^{n,s,v}(x-y)f(y)dy\,\chi_Q\\
&\lc\sum_{Q\in\Q_{n-j+s}}\frac1{|Q|}\int_{Q}f(y)dy\,\chi_Q\,2^{-jd+(n+s)\iota}\|\Om\|_{1}\\
&=2^{-jd+(n+s)\iota}\|\Om\|_{1}f_{n-j+s}.
\end{split}
\Ee

Now applying the Hilbert-valued H\"older inequality in Lemma \ref{l:28holder}, we get
\Be\label{e:19tracef}
\begin{split}
\tau\Big(\Big|p_{n-j}&\big(K_j^{n,s,v}(x-\cdot)f(\cdot)\big)_{n-j+s}p_{n-j+s}(z)\Big|^2\Big)\\
&\leq\tau\Big(p_{n-j}\big( K_j^{n,s,v}(x-\cdot)f(\cdot)\big)_{n-j+s}p_{n-j}(z)\Big)\\
&\quad\quad\times\|p_{n-j+s}\big(K_j^{n,s,v}(x-\cdot)f(\cdot)\big)_{n-j+s}p_{n-j+s}(z)\|_\M.\\
\end{split}\Ee
Applying \eqref{e:28Kjaverage} and the noncommutative Calder\'on--Zygmund property \eqref{e:28PQ}, the proceeding inequality is majorized by
\Bes
\begin{split}
2^{-2jd+2(n+s)\iota}\|\Om\|^2_{1}\tau\big(p_{n-j}&f_{n-j+s}p_{n-j}(z)\big)\|p_{n-j+s}f_{n-j+s}p_{n-j+s}\|_\A\\
&\lc2^{-2jd+2(n+s)\iota}\|\Om\|_{1}\lam\tau\big(p_{n-j}f_{n-j+s}p_{n-j}(z)\big).
\end{split}
\Ees

By the definition of $\mathbf{E}_j^{n,s,v}(x)$, we see that $\mathbf{E}^{n,s,v}_{j}(x)$ is supported in a rectangle with center $x$, one sidelength at most $2^{j+2}$ and $d-1$ sidelengths at most $2^{j+2-(n+s)\gamma}$. Then for any fixed $z\in\R^d$, a simple geometric observation shows that the set $\{x:\,\mathbf{E}_j^{n,s,v}(x)\ni z\}$ is contained in a rectangle with center $z$, one sidelength at most $2^{j+2}$ and $d-1$ sidelengths at most $2^{j+2-(n+s)\gamma}$. Hence, we have the following estimate
\Be\label{e:19ejnsv}
\Big|\int_{\{x:\,\mathbf{E}_j^{n,s,v}(x)\ni z\}}dx\Big|\lc2^{jd-{(n+s)}\ga(d-1)}.
\Ee
Plugging \eqref{e:19tracef} into \eqref{e:19bl2fn}, then applying the Fubini theorem with \eqref{e:19ejnsv}, and finally using the trace preserving property of expectation, we obtain
\Bes
\begin{split}
II&\lc (s+1)\sum_j\|T_j^{n,s,v}p_{n-j}fp_{n-j+s}\|^2_{L_2(\A)}\\
&\lc (s+1)2^{-2(n+s)\ga (d-1)+2(n+s)\iota}\|\Om\|_{1}\lam\sum_{j\in\Z}\vphi\big(p_{n-j}f_{n-j+s}p_{n-j}\big)\\
&\lc (s+1)2^{-2(n+s)\ga (d-1)+2(n+s)\iota}\|\Om\|_{1}\lam\sum_{j\in\Z}\vphi\big(p_{n-j}fp_{n-j}\big)\\
&\lc (s+1)2^{-2(n+s)\ga (d-1)+2(n+s)\iota}\|\Om\|_{1}\lam\|f\|_{L_1(\A)},
\end{split}
\Ees
which is the desired estimate for $II$. Hence, combining the estimates for $I$ and $II$, we complete the proof of the third inequality in Lemma \ref{L:28fnjs}.

Finally, we point out that the proofs of the first and second inequalities in Lemma \ref{L:28fnjs} follow step by step from that of the third inequality, with the only modification being the replacement of \eqref{e:28pnjin} by \eqref{e:28pnjinaverage}.
\end{proof}

\vskip0.24cm
\medskip

\section*{Declarations}

\subsection*{Competing interests} The author declares no competing interests.

\subsection*{Funding} This work is supported by National Natural Science Foundation of China (No. 12271124, No. 12322107 and No. W2441002) and Heilongjiang Provincial Natural Science Foundation of China (YQ2022A005).

\subsection*{Data Availability} No datasets were generated or analysed during the current study.


\vskip1cm

\bibliographystyle{amsplain}
\bibliography{rf}

\end{document}